\documentclass[a4paper,12pt]{article}

\usepackage[utf8]{inputenc}
\usepackage[english]{babel}
\usepackage[hmargin=2.5cm,vmargin=3cm]{geometry}
\usepackage{amsmath,amssymb,amsthm,mathtools}
\usepackage{graphicx}
\usepackage{booktabs,array}
\usepackage{enumitem}
\usepackage{float}
\usepackage{subcaption}
\usepackage{xcolor}
\usepackage{titlesec}
\usepackage{tikz}
\usepackage{authblk}
\usetikzlibrary{arrows.meta,calc,positioning}
\usepackage[colorlinks=true,linkcolor=blue,citecolor=blue]{hyperref}
\usepackage[nameinlink]{cleveref}
\usepackage[square,sort,comma,numbers]{natbib}

\allowdisplaybreaks[2]
\usepackage{pgfplots}
\pgfplotsset{compat=1.18}
\theoremstyle{plain}
\newtheorem{thm}{Theorem}[section]
\newtheorem{theorem}[thm]{Theorem}
\newtheorem{lemma}[thm]{Lemma}
\newtheorem{corollary}[thm]{Corollary}

\newtheorem{conjecture}[thm]{Conjecture}

\newtheorem{remark}[thm]{Remark}

\tikzset{
  vertex/.style={circle,draw,thick,minimum size=5mm,inner sep=0pt},
  deleted/.style={circle,draw,thick,dashed,minimum size=5mm,inner sep=0pt},
  edge/.style={thick}
}
\providecommand{\keywordsname}{Keywords}
\newenvironment{keywords}
  {\par\medskip\noindent\textbf{\keywordsname:}\ }
  {\par\medskip}
\newcommand{\Spec}{\operatorname{Spec}}
\newcommand{\mad}{\operatorname{mad}}
\title{Sharp Upper Bounds for the Median Eigenvalues of Graphs     \thanks{This work is partly supported by the National Natural Science Foundation of China (No.12371354,
W2521102), the Montenegrin-Chinese Science and Technology Cooperation Project (No.4-3) and
the Science and Technology Commission of Shanghai Municipality (No.25LN3200600).}}

\author[a]{Zhengbo Chen}
\author[b]{Yuzhenni Wang}
\author[a]{Xiao-Dong Zhang}

\affil[a]{School of Mathematical Sciences, MOE-LSC, SHL-MAC,
Shanghai Jiao Tong University, Shanghai 200240, China}

\affil[b]{School of Statistics and Data Sciences,
Shanghai University of International Business and Economics,
Shanghai 201620, China}

\usepackage{fancyhdr}

\date{}
\begin{document}
\maketitle
\renewcommand{\thefootnote}{}

	\footnote{Email addresses:
    \texttt{czb911@sjtu.edu.cn} (Z. Chen);
    \texttt{wangyuzhenni@suibe.edu.cn} (Corresponding author: Y. Wang);
    \texttt{xiaodong@sjtu.edu.cn} (X.-D. Zhang)}

\begin{abstract}
Let $\lambda_1\geq\lambda_2\geq\cdots\geq\lambda_n$ be the eigenvalues of a simple graph $G$ of order $n$.
The HL-index of $G$ is defined by $R(G)=\max\{|\lambda_h|,|\lambda_\ell|\}$ with $h=\lfloor(n+1)/2\rfloor$ and $\ell=\lceil(n+1)/2\rceil$. 
In this paper, we prove that if $G$ is $ K_4$-minor-free or $ K _ {2,3} $-minor-free, then  $R(G)\leq\sqrt{5}-1$ with equality attained by an infinite family of outerplanar graphs. 
Moreover, we show that 
$R(G)\leq\sqrt{d-2}$ for triangle-free graphs with maximum degree at most $d$ and average degree at most $(d-2)(d^2-2d+2)/(d^2-3d+5)$.
\end{abstract}

\section{Introduction}
In mathematical chemistry,  the HOMO--LUMO separation, which is the gap between the Highest Occupied Molecular Orbital(HOMO) and Lowest Unoccupied Molecular Orbital (LUMO), is linearly related to the median eigenvalues of a graph.
Fowler and Pisanski \citep{FowlerPisanski2010Chemical,FowlerPisanski2010Fullerenes}
introduced the notion of the HL-index of a graph (see also Jakli\u{c} et al. \citep{JaklicFowlerPisanski2012}).
Mohar\citep{Mohar2015HL,Mohar2016BipartiteSubcubic,Mohar2013BipartitePlanar}  systematically studied the HL-index of graphs with bounded maximum degree and the associated median eigenvalue problem. Since then, median eigenvalues have attracted considerable attention, particularly their bounds and connections with graph structure, owing to their significance in spectral graph theory and mathematical chemistry.

Let $G=(V(G), E(G))$ be a simple graph of order $n$ with vertex set $V(G)$ and edge set $E(G)$. 
Let $A(G)$ denote its adjacency matrix. The characteristic polynomial of $G$ is defined by $ \chi_G(t)=\det\bigl(tI-A(G)\bigr). $ Its roots, counted with multiplicities, are the adjacency eigenvalues of $G$. We order them as $\lambda_1(G)\geq\lambda_2(G)\geq\cdots\geq\lambda_n(G).$
The HL-index of $G$ is defined as
\[R(G)=\max\{|\lambda_h(G)|,|\lambda_\ell(G)|\},
\]
where $h=\left\lfloor\frac{n+1}{2}\right\rfloor$ and $\ell=\left\lceil\frac{n+1}{2}\right\rceil.$
Note that $h=\ell$ when $n$ is odd.

The degree of a vertex $v$ in a graph $G$, denoted by $d_G(v)$, is defined as the number of edges incident to $v$. 
For a nonempty graph $G$, the maximum degree of $G$ is defined by $\Delta(G)=\max_{v\in V(G)} d_G(v)$, and 
the average degree of $G$ is defined by $\overline d(G)=2|E(G)|/|V(G)|.$
The maximum average degree of $G$ is defined by
\[
\mad(G)=\max_{\varnothing\neq H\subseteq G}
\frac{2|E(H)|}{|V(H)|}.
\] 
A graph is called  {\it subcubic} if its maximum degree is at most 3. A graph $H$ is a {\it minor} (or {\it $H$-minor}) of a graph $G$ if a copy of $H$ can be obtained from $G$ 
by deleting vertices and edges or contracting edges of $G$. 
A graph is {\it $H$-minor-free} if $H$ is not a minor of it.

Mohar \citep{Mohar2015HL} proved that every subcubic graph satisfies
$R(G)\leq\sqrt2$. Furthermore, Mohar proposed the following conjecture.
\begin{conjecture}\citep{Mohar2015HL}\label{con1}
For every subcubic planar graph $G$, $R(G)\le 1$.
\end{conjecture}
Mohar  \citep{Mohar2013BipartitePlanar} confirmed that Conjecture \ref{con1} holds for  bipartite subcubic planar graphs. Later,  Mohar \citep{Mohar2016BipartiteSubcubic} proved that $R(G)\le 1$ holds for every bipartite subcubic graph $G$ except the Heawood graph, whose median eigenvalues are $\pm \sqrt{2}$.  Benediktovich \cite{Benediktovich2014} confirmed Conjecture \ref{con1} for subcubic outerplanar graphs. 
Wang and Zhang \citep{WangZhang2024} confirmed Conjecture \ref{con1} for every $K_4$-minor-free subcubic graph and for every subcubic graph containing $K_{2,3}$ as a subgraph. Chen, Wang, and Zhang \citep{ChenWangZhang2025} confirmed Conjecture \ref{con1} for every graph of girth at least $8$, and proved $R(G)\le 1$ for every subcubic graph with maximum average degree less than 
$\frac{44}{17}$. More recently, Acharya, Jeter, and Jiang \citep{AcharyaJeterJiang2025} completely confirmed Conjecture \ref{con1}. In fact, they proved the stronger result: every connected subcubic graph other than the Heawood graph satisfies $R(G)\leq1$.

It is natural to consider graphs with maximum degree $d$ for a general integer $d$.
Mohar \citep{Mohar2015HL} proved that $R(G)\leq\sqrt d$ for every graph $G$ with maximum degree $d$. 
Furthermore, Mohar proposed the following conjecture.
\begin{conjecture}\citep{Mohar2015HL}\label{con2}
For every graph $G$ with maximum degree $d$, $R(G)\le \sqrt{d-1}$.
\end{conjecture}
Mohar \citep{Mohar2015HL} also pointed out that the HL-index of the incidence graph of a projective plane of order $d-1$ is equal to $\sqrt{d-1}$; hence the bound would be optimal whenever $d-1$ is a prime power. 
Acharya, Jiang, and Zhang \citep{AcharyaJiangZhang2026} proved that
$\lambda_h(G)\leq\sqrt{d-1}$ for every graph of maximum degree at most $d$. Furthermore, they \citep{AcharyaJiangZhang2026} confirmed Conjecture \ref{con2} for all triangle-free graphs. Moreover, they \citep{AcharyaJiangZhang2026} proved that $R(G)\leq\sqrt{\overline d(G)}$ for every graph $G$ with average degree $\overline d(G)$.

Motivated by the above conjectures and results, we study the relationship between graph structure and the HL-index in this paper.
The following is the first main result of this work.

\begin{theorem}
\label{thm:intro-K4}
Let $G$ be a $K_4$-minor-free graph or a $K_{2,3}$-minor-free graph. Then
$R(G)\leq\sqrt5-1$ with equality attained by an infinite family of outerplanar graphs. 
\end{theorem}

Theorem~\ref{thm:intro-K4} extends the result of Wang and Zhang
\citep{WangZhang2024} from subcubic $K_4$-minor-free graphs to all the
 $K_4$-minor-free graphs, without any restriction on the maximum degree.
 Our second main result concerns graphs with bounded maximum degree.
\begin{theorem}
\label{thm:intro-mad-subquartic}
Let $G$ be a simple graph with 
$\Delta(G)\leq4 $
and
$\operatorname{mad}(G)\leq\frac83.$
Then $R(G)\leq\sqrt2.$
\end{theorem}
Theorem~\ref{thm:intro-mad-subquartic} extends the result of Chen, Wang, and Zhang \citep{ChenWangZhang2025} from subcubic graphs to graphs with maximum degree at most $4$. Motivated by the proof of \cite[Theorem~1.1]{AcharyaJiangZhang2026}, which shows that
$\lambda_h(G)\leq \sqrt{d-1}$ for every graph $G$ of maximum degree at most $d$, we use a similar approach to establish another main result.

\begin{theorem}
\label{thm:intro-average-degree}
Let $d\geq3$ be a positive integer, and let $G$ be a simple graph satisfying
$\Delta(G)\leq d$ 
and
$\overline d(G)\leq\tau_d$ with 
$\tau_d=(d-2)(d^2-2d+2)/(d^2-3d+5).$ 
Then we have (i) $\lambda_h(G)\leq\sqrt{d-2}$; (ii) if $G$ is additionally triangle-free, then
$R(G)\leq\sqrt{d-2}.$
\end{theorem}

Thus, under an explicit average-degree condition, the upper bound
$\sqrt{d-1}$ of Acharya, Jiang, and Zhang
\citep{AcharyaJiangZhang2026} is improved to $\sqrt{d-2}$. In
particular, when $d=4$, the condition
$\overline d(G)\leq20/9$ gives $\lambda_h(G)\leq\sqrt2$, and gives
$R(G)\leq\sqrt2$ when $G$ is triangle-free.

The rest of the paper is organized as follows. In the next section,  some technical preliminary results are provided. In Section 3, we prove Theorem \ref{thm:intro-K4}, and in Section 4 we prove Theorem \ref{thm:intro-mad-subquartic}. The proof of Theorem \ref{thm:intro-average-degree} is given in Section 5. 
\section{Preliminaries}
\label{sec:preliminaries}

Throughout the paper, $|G|$ denotes $|V(G)|$, and
$[k]=\{1,\ldots,k\}$ for every positive integer $k$. For two
vertex sets $A,B\subseteq V(G)$, let $e(A,B)$ denote the number of
edges with one end in $A$ and the other in $B$. 
For a graph $H$ and a real number $b>0$, define
\[
n_b^+(H)=|\{i:\lambda_i(H)>b\}|,
\qquad
n_b^-(H)=|\{i:\lambda_i(H)<-b\}|.
\]

First, we need the following well-known theorems in spectral graph theory (for example, see  \cite[pp. 17--18]{C2010}).
\begin{theorem}[Eigenvalue Interlacing Theorem]
\label{thm:interlacing}
Let $G$ be an $n$-vertex graph, let $S\subseteq V(G)$, and let 
$s=|S|$. Then, for every $1\leq i\leq n-s$,
\[
\lambda_i(G)
\geq
\lambda_i(G-S)
\geq
\lambda_{i+s}(G)
\quad\text{and}\quad
\lambda_{n-i+1}(G)
\leq
\lambda_{n-i+1}(G-S)
\leq
\lambda_{n-s-i+1}(G).
\]
\end{theorem}

\begin{lemma}
\label{lem:count-characterization}
Let $H$ be a graph of order $m$, and let $b$ be a positive number.
Then $R(H)\leq b$ if and only if
\[
n_b^{\pm}(H)\leq\left\lfloor\frac{m-1}{2}\right\rfloor.
% \qquad\text{and}\qquad
% n_b^-(H)\leq\left\lfloor\frac{m-1}{2}\right\rfloor.
\]
\end{lemma}

\begin{proof}
By definition, $R(H)\leq b$ if and only if
$\lambda_h(H)\leq b$ and $\lambda_\ell(H)\geq-b$.
These two inequalities are respectively equivalent to
$n_b^+(H)\leq h-1$ and $n_b^-(H)\leq m-\ell$.
Since $h-1=m-\ell=\lfloor(m-1)/2\rfloor$, the result follows.
\end{proof}

\begin{lemma}
\label{lem:count-interlacing}
For every graph $G$, every $S\subseteq V(G)$, and every $b>0$,
\[
n_b^\pm(G)\leq n_b^\pm(G-S)+|S|.
\]
\end{lemma}

\begin{proof}
This follows immediately from
Theorem~\ref{thm:interlacing}, since deleting $|S|$ vertices can
shift the index of an eigenvalue by at most $|S|$.
\end{proof}

\begin{lemma}
\label{lem:general-deletion}
Let $b>0$, let $G$ be an $n$-vertex graph, and let
$S\subseteq V(G)$ with $|S|=s$. Let $C_1,\ldots,C_k$ be the
connected components of $G-S$. Suppose that
$R(C_i)\leq b$
for $i\in[k]$
and
\[
\sum_{i=1}^k
\left\lfloor\frac{|C_i|-1}{2}\right\rfloor
\leq
\left\lfloor\frac{n-1}{2}\right\rfloor-s.
\]
Then $R(G)\leq b$.
\end{lemma}

\begin{proof}
By Lemma~\ref{lem:count-characterization},
\[
n_b^\pm(C_i)
\leq
\left\lfloor\frac{|C_i|-1}{2}\right\rfloor
\qquad \text{for} \quad i\in[k].
\]
Since the spectrum of a disconnected graph is the multiset union of
the spectra of its components,
\[
\begin{aligned}
n_b^\pm(G-S)
=\sum_{i=1}^k n_b^\pm(C_i)
\leq
\sum_{i=1}^k
\left\lfloor\frac{|C_i|-1}{2}\right\rfloor
\leq
\left\lfloor\frac{n-1}{2}\right\rfloor-s.
\end{aligned}
\]
Lemma~\ref{lem:count-interlacing} now gives
\[
n_b^\pm(G)
\leq
\left\lfloor\frac{n-1}{2}\right\rfloor,
\]
and the result follows from
Lemma~\ref{lem:count-characterization}.
\end{proof}

The following theorem reduces the study of the HL-index of a graph to that of suitable subgraphs.

\begin{theorem}   

\label{lem:two-cut-reduction}
Let $b$ be a positive number, let $G$ be an $n$-vertex graph, and let
$S\subseteq V(G)$ with $|S|=2$. Suppose that every component $C$ of
$G-S$ satisfies $R(C)\leq b$. Then $R(G)\leq b$ whenever one of the
following holds:
\begin{enumerate}[label=\rm(\roman*)]
\item $n$ is odd and $G-S$ is disconnected;
\item $n$ is even and $G-S$ has at least three components;
\item $n$ is even and $G-S$ has exactly two components, both of even
order.
\end{enumerate}
\end{theorem}

\begin{proof}
Let $C_1,\ldots,C_k$ be the components of $G$ of orders $m_1,...m_k$. 

\textbf{Case \rm(i).} By Lemma~\ref{lem:general-deletion}, we have $R(G)\leq b$, since
\[
\sum_{i=1}^k
\left\lfloor\frac{m_i-1}{2}\right\rfloor
\leq
\left\lfloor\frac{n-2-k}{2}\right\rfloor
\leq
\left\lfloor\frac{n-1}{2}\right\rfloor-2.
\]

\textbf{Case~\rm(ii).}
It follows by the same argument as in Case (i).

\textbf{Case~\rm(iii).}  Let $m_1=2r$ and $m_2=2s$.
Since $r+s=n/2-1$,
\[
(r-1)+(s-1)
=\frac n2-3
=\left\lfloor\frac{n-1}{2}\right\rfloor-2.
\]
Thus the result follows from Lemma~\ref{lem:general-deletion}.
\end{proof}

In order to prove Theorem \ref{thm:intro-K4}, we determine the exact value of the HL-index for a class of graphs. 
For every integer $m\geq3$, let $T_m$ be the graph with
\[
V(T_m)=\{x_i,y_i:i\in\mathbb Z_m\},
\qquad
E(T_m)=\{x_ix_{i+1},x_iy_i,y_ix_{i+1}:i\in\mathbb Z_m\},
\]
where the indices are taken modulo $m$. Thus
$x_1x_2\cdots x_mx_1$ is a cycle and
$N_{T_m}(y_i)=\{x_i,x_{i+1}\}$ for every $i\in\mathbb Z_m$.
The graph $T_6$ is shown in Figure \ref{fig:Tm}. 
\begin{figure}[htbp]
\centering
\begin{tikzpicture}[
  scale=0.82,
  vtx/.style={circle,fill=black,inner sep=2pt},
  edge/.style={thick}
]
  \def\rin{2.15}
  \def\rout{3.25}

  \foreach \i in {1,...,6}{
    \pgfmathsetmacro{\anglex}{90-60*(\i-1)}
    \pgfmathsetmacro{\angley}{90-60*(\i-0.5)}

    \node[vtx] (x\i) at (\anglex:\rin) {};
    \node[vtx] (y\i) at (\angley:\rout) {};

    \node at (\anglex:1.70) {$x_{\i}$};
    \node at (\angley:3.62) {$y_{\i}$};
  }

  \foreach \i in {1,...,6}{
    \pgfmathtruncatemacro{\j}{mod(\i,6)+1}
    \draw[edge] (x\i)--(x\j);
    \draw[edge] (x\i)--(y\i);
    \draw[edge] (y\i)--(x\j);
  }
\end{tikzpicture}
\caption{The graph $T_6$.}
\label{fig:Tm}
\end{figure}
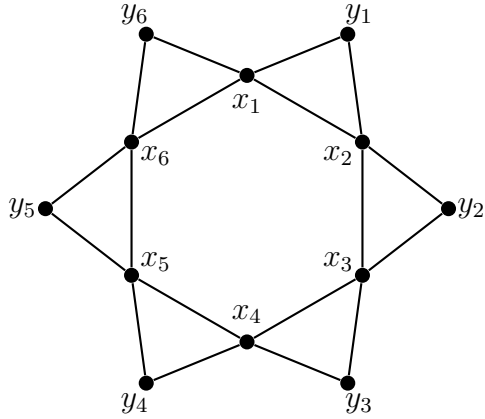

Note that every graph $T_m$ with $m\ge 3$ is outerplanar and hence
both $K_{4}$-minor-free and  $K_{2,3}$-minor-free. 
\begin{lemma}
\label{lem:Tm-HL-index}
For every $m\geq3$, the outerplanar graph $T_m$ satisfies 
$R(T_m)=\sqrt5-1.$
\end{lemma}
\begin{proof}
Let $C$ be the adjacency matrix of the cycle
$x_1x_2\cdots x_mx_1$, and let $B$ be the bipartite adjacency matrix
between $\{x_1,\ldots,x_m\}$ and $\{y_1,\ldots,y_m\}$. Then
\[
A(T_m)=
\begin{pmatrix}
C&B\\
B^{\mathsf T}&0
\end{pmatrix}.
\]

Every $x_i$ has exactly two neighbors among the vertices $y_j$, and
two distinct vertices $x_i,x_j$ have a common neighbor among the vertices 
$y_j$ precisely when they are consecutive on the cycle. Hence
$BB^{\mathsf T}=2I+C.$

For $t\neq0$, the Schur complement formula gives
\[
\begin{aligned}
\chi_{T_m}(t)
&=
\det
\begin{pmatrix}
tI-C&-B\\
-B^{\mathsf T}&tI
\end{pmatrix} \\
&=
t^m\det\left(tI-C-\frac1tBB^{\mathsf T}\right) 
=
\det\left(t^2I-tC-BB^{\mathsf T}\right).
\end{aligned}
\]
Using $BB^{\mathsf T}=2I+C$, we obtain
\[
\chi_{T_m}(t)
=
\det\left(t^2I-tC-(2I+C)\right).
\]

The eigenvalues of $C$ are
\[
\theta_k=2\cos\frac{2\pi k}{m},
\qquad\text{for } 0\leq k<m.
\]
It follows that
\[
\chi_{T_m}(t)
=
\prod_{k=0}^{m-1}
\left(t^2-\theta_kt-(2+\theta_k)\right).
\]
For $\theta\in[-2,2]$, let
\[
r_\pm(\theta)
=
\frac{\theta\pm\sqrt{(\theta+2)^2+4}}{2}.
\]
Thus $r_+(\theta)$ and $r_-(\theta)$ are the two roots of
$t^2-\theta t-(2+\theta)=0$. Moreover,
\[
r_+(\theta)\geq0,
\qquad
r_-(\theta)<0,
\]
and both functions are strictly increasing on $[-2,2]$, since
\[
r_\pm'(\theta)
=
\frac12\left(
1\pm\frac{\theta+2}{\sqrt{(\theta+2)^2+4}}
\right)>0.
\]

Consequently, the $m$ values $r_+(\theta_k)$ are the nonnegative
eigenvalues of $T_m$, while the $m$ values $r_-(\theta_k)$ are its
negative eigenvalues. If $\theta_{\min}$ denotes the least
eigenvalue of $C_m$, then
\[
\lambda_m(T_m)=r_+(\theta_{\min}),
\qquad
\lambda_{m+1}(T_m)=r_-(2).
\]
Indeed, $2$ is the largest eigenvalue of $C_m$, and both $r_+$ and
$r_-$ are increasing.

We have
\[
\lambda_{m+1}(T_m)=r_-(2)=1-\sqrt5.
\]
Furthermore,
\[
\theta_{\min}
=
\begin{cases}
-2, & \text{if $m$ is even},\\[1mm]
-2\cos(\pi/m), & \text{if $m$ is odd},
\end{cases}
\]
so $\theta_{\min}\leq-1$ for every $m\geq3$. Hence
\[
0\leq\lambda_m(T_m)
=r_+(\theta_{\min})
\leq r_+(-1)
=\frac{\sqrt5-1}{2}.
\]
Therefore,
\[
\begin{aligned}
R(T_m)
=
\max\{\lambda_m(T_m),-\lambda_{m+1}(T_m)\}
=
\sqrt5-1.
\end{aligned}
\]
This proves Lemma~\ref{lem:Tm-HL-index}.
\end{proof}

\section{The proof of Theorem~\ref{thm:intro-K4}}
\label{sec:K4}

\subsection{$K_{4}$-minor-free graphs}
\label{subsec:K4}
In order to prove Theorem \ref{thm:intro-K4}, we first present the following lemma and theorem.

\begin{lemma}
\label{lem:alternating-Hamilton-cycle}
Let $m\geq3$, and suppose that $G$ has a Hamilton cycle
$x_1y_1x_2y_2\cdots x_my_mx_1$ such that
$d_G(y_i)=2$ for every $i\in\mathbb Z_m$. Then
$
R(G)=\sqrt5-1
\text{ if }G\cong T_m;
$ 
$R(G)\leq1
\text{, otherwise}.$

\end{lemma}

\begin{proof}
Let $X=\{x_1,\ldots,x_m\}$, $Y=\{y_1,\ldots,y_m\}$, and
$F=G[X]$. Call an edge $x_ix_{i+1}$ of $F$ a \emph{short edge};
every other edge of $F$ is called a \emph{long edge}.

Suppose first that $F$ excludes a short edge, say $x_ix_{i+1}$. Let
$H=G[Y\cup\{x_i,x_{i+1}\}].$
Then
$H\cong P_5\cup(m-3)K_1,$
where the nontrivial component is
$y_{i-1}x_iy_ix_{i+1}y_{i+1}$. Since
$\Spec(P_5)=\{\sqrt3,1,0,-1,-\sqrt3\},$
we have $n_1^+(H)=n_1^-(H)=1$. The graph $H$ is obtained from $G$
by deleting $m-2$ vertices, and therefore
\[
n_1^\pm(G)
\leq n_1^\pm(H)+m-2
=m-1
=\left\lfloor\frac{2m-1}{2}\right\rfloor.
\]
Lemma~\ref{lem:count-characterization} gives $R(G)\leq1$.

Suppose next that $F$ contains a long edge $x_ix_j$. Then $x_i$ and
$x_j$ are nonconsecutive, so their four neighbors in $Y$ are
distinct. For
$H=G[Y\cup\{x_i,x_j\}],$
we have
$H\cong D\cup(m-4)K_1,$
where $D$ is the double star obtained by joining the centers of two
copies of $K_{1,2}$. Since
$\chi_D(t)=t^2(t^2-1)(t^2-4), 
\Spec(D)=\{2,1,0,0,-1,-2\},$
we again have $n_1^+(H)=n_1^-(H)=1$. Hence
\[
n_1^\pm(G)
\leq1+m-2
=m-1
=\left\lfloor\frac{2m-1}{2}\right\rfloor,
\]
and Lemma~\ref{lem:count-characterization} yields $R(G)\leq1$.

It remains to consider the case in which $F$ contains every short
edge and excludes every long edge. In this case $G\cong T_m$, and
Lemma~\ref{lem:Tm-HL-index} gives
$R(G)=R(T_m)=\sqrt5-1.$
\end{proof}

\begin{theorem}
\label{thm:K4-main}
For every $K_4$-minor-free graph $G$, 
$R(G)\leq\sqrt5-1.$
\end{theorem}

\begin{proof}
Suppose otherwise, and let $G$ be a counterexample of minimum order.
By minimality and Lemma~\ref{lem:general-deletion}, applied with
$S=\varnothing$ to the components of $G$, the graph $G$ is connected.
Similarly, applying Lemma~\ref{lem:general-deletion} with
$S=\{v\}$ excludes every cut vertex $v$ of $G$. Since
$R(K_1)=0$ and $R(K_2)=1<\sqrt5-1$, the graph $G$ has order at least
$3$ and is therefore $2$-connected.

Let $C$ be a longest cycle of $G$. If $G=C$, then
\[
\Spec(C_n)=
\left\{2\cos\frac{2\pi j}{n}:0\leq j<n\right\}
\]
gives $R(G)\leq1<\sqrt5-1$. Thus $G$ is not a cycle.

\medskip
\noindent
\textbf{Claim 1.} Every vertex of $G$ lies on $C$.

\begin{proof}
Suppose, to the contrary, that $V(G)\setminus V(C)\neq\varnothing$,
and let $D$ be a component of $G-V(C)$. Since $G$ is
$2$-connected, $D$ has at least two distinct neighbors on $C$.
Choose distinct vertices
$
u,v\in N_G(D)\cap V(C),
$
and let $P_1$ be a $uv$-path whose internal vertices lie in $D$.
Let $P_2$ and $P_3$ be the two $uv$-arcs of $C$.

The vertices $u$ and $v$ are not consecutive on $C$. Indeed, if
$uv\in E(C)$, then replacing the edge $uv$ by $P_1$ would produce
a cycle longer than $C$. Hence
$
P_1,P_2, P_3
$
are three internally disjoint $uv$-paths, each with nonempty
interior.

We claim that the interiors of $P_1,P_2,P_3$ lie in three distinct
components of $G-\{u,v\}$. Suppose otherwise. Then the interiors of
at least two of these paths are joined in $G-\{u,v\}$.

For each pair $\{i,j\}\subseteq\{1,2,3\}$ for which
$V(P_i)\setminus\{u,v\}$ and $V(P_j)\setminus\{u,v\}$ are joined in
$G-\{u,v\}$, choose a shortest path $Q_{ij}$ in $G-\{u,v\}$ with 
one end in $V(P_i)\setminus\{u,v\}$ and the other end in
$V(P_j)\setminus\{u,v\}$. Among all such paths, choose one of minimum
length. After relabeling $P_1,P_2,P_3$ if necessary, denote this path
by $Q=Q_{12}$, and let its end vertices be
$a\in V(P_1)\setminus\{u,v\}$ and
$b\in V(P_2)\setminus\{u,v\}$.

By the choice of $Q$, no internal vertex of $Q$ lies on $P_1\cup P_2\cup P_3$, since otherwise a proper subpath of $Q$ would be a shorter path joining two of their interiors. Hence $Q$, $P_3$, and the four relevant subpaths of $P_1$ and $P_2$ form a subdivision of $K_4$. Hence $G$ contains a $K_4$-minor, contradicting the assumption that $G$ is $K_4$-minor-free.

Therefore, the interiors of $P_1,P_2,P_3$ lie in three distinct
components of $G-\{u,v\}$. In particular,
$G-\{u,v\}$ has at least three components. Every such component is
a proper $K_4$-minor-free subgraph of $G$, and hence has HL-index at
most $\sqrt{5}-1$ by the minimality of $G$. If $|G|$ is odd, apply
Theorem~\ref{lem:two-cut-reduction}(i); if $|G|$ is even, apply
Theorem~\ref{lem:two-cut-reduction}(ii). In either case,
$
R(G)\leq\sqrt{5}-1,$
a contradiction. This proves Claim~1.
\end{proof}

\medskip
\noindent

\medskip
\noindent
\textbf{Claim 2.} The end vertices of every chord of $C$ form a
$2$-vertex cut.

\begin{proof}
Let $xy$ be a chord of $C$. Since $xy$ is not an edge of $C$, the
two open $xy$-arcs of $C$ are both nonempty. Denote their vertex
sets by $A$ and $B$.

Suppose that $G-\{x,y\}$ is connected. By Claim~1, every vertex of
$G$ lies on $C$. Since the subgraphs of $C-\{x,y\}$ induced by
$A$ and $B$ are connected, the connectedness of $G-\{x,y\}$
implies that there exists an edge
$ab\in E(G)$
with
$a\in A, b\in B.$
The edge $ab$ is necessarily a chord of $C$. The two chords $xy$ and $ab$, together with the four subpaths of $C$ between successive vertices among $x,a,y,b$, form a subdivision of $K_4$. Hence $G$ contains a $K_4$-minor, contradicting the assumption that $G$ is $K_4$-minor-free. Therefore, $G-\{x,y\}$ is disconnected, and Claim~2 follows.
\end{proof}

\medskip
\noindent
\textbf{Claim 3.} The order of $G$ is even, and every chord of $C=v_1v_2\cdots v_nv_1$
joins two vertices whose indices have the same parity.

\begin{proof}
By Claim~1, every vertex of $G$ lies on $C$. Since $G$ is not a
cycle, $C$ has at least one chord.

Suppose first that $n$ is odd. By Claim~2, deleting the
end vertices of any chord disconnects $G$. Every component of the
resulting graph is a proper $K_4$-minor-free subgraph of $G$, and
hence has HL-index at most $\sqrt5-1$ by the minimality of $G$.
Theorem~\ref{lem:two-cut-reduction}(i) then gives
$R(G)\leq\sqrt5-1,$ 
a contradiction. Therefore,
$n=2m$ is even.

It remains to prove the assertion about the parity of the indices of 
end vertices of a chord. Suppose that $v_iv_j$ is a chord with $i$
and $j$ of opposite parity. Since $n$ is even, both open
$v_iv_j$-arcs of $C$ have even order.

By Claim~2, the graph $G-\{v_i,v_j\}$ is disconnected. By Claim~1,
the two open $v_iv_j$-arcs partition
$V(G)\setminus\{v_i,v_j\},$
and each of them induces a connected subgraph. Consequently, they
are precisely the two components of $G-\{v_i,v_j\}$. Both
components have even order, so
Theorem~\ref{lem:two-cut-reduction}(iii) gives
$R(G)\leq\sqrt5-1,$ again a contradiction.

Therefore, every chord of $C$ joins two vertices whose indices
have the same parity, and Claim~3 follows.
\end{proof}

Let $O$ be the set of chords with two odd-indexed end vertices, and
let $E$ be the set of chords with two even-indexed end vertices.

\medskip
\noindent
\textbf{Claim 4.} At least one of $O$ and $E$ is empty.

\begin{proof}
First, note that no two chords of $C$ can cross. More precisely, if two chords have
four distinct end vertices occurring alternately on $C$, then the
two chords together with the four intervening subpaths of $C$ form
a subdivision of $K_4$. Thus no such pair of chords exists.

Suppose, to the contrary, that both $O$ and $E$ are nonempty. For
$p,q\in V(C)$, let
$d_C(p,q)$
denote the number of edges in a shortest $pq$-arc of $C$. 
For two chords $e$ and $f$, define
\[
d_C(e,f)
=
\min\bigl\{
d_C(p,q):
p\text{ is an end vertex of }e,\ 
q\text{ is an end vertex of }f
\bigr\}.
\]

Choose chords
$e=xx'\in O$
and
$f=uu'\in E$
such that $d_C(e,f)$ is minimum. Relabel the end vertices if
necessary so that
$d_C(x,u)=d_C(e,f)=d,$
and let $P$ be a shortest $xu$-arc of $C$. Since $x$ and $u$ have
opposite parity, $d$ is odd.

Suppose first that $d>1$. No internal vertex of $P$ is incident
with a chord. Indeed, let $z$ be an internal vertex of $P$ and
suppose that $z$ is incident with a chord $g$. If $z$ has the same
parity as $x$, then $g\in O$, and
\[
d_C(g,f)\leq d_C(z,u)<d,
\]
contrary to the choice of $e$ and $f$. If $z$ has the same parity
as $u$, then $g\in E$, and
\[
d_C(e,g)\leq d_C(x,z)<d,
\]
which is again a contradiction.

It follows that
$V(P)\setminus\{x,u\}$
is the vertex set of a component of $G-\{x,u\}$. This component has
order $d-1$, which is even. The remaining vertices of
$G-\{x,u\}$ also have even total order, since $|G|$ is even and
$d$ is odd. If $G-\{x,u\}$ has at least three components, then
Theorem~\ref{lem:two-cut-reduction}(ii) applies. If it has exactly
two components, both have even order and
Theorem~\ref{lem:two-cut-reduction}(iii) applies. In either case,
$R(G)\leq\sqrt5-1$, a contradiction.

It remains to consider the case
$d=1.$
Thus $x$ and $u$ are consecutive on $C$. Among all pairs
$e=xx'\in O$ and $f=uu'\in E$ satisfying $d_C(x,u)=1$, choose
one for which the following gap is minimum. Orient $C$ from $x$
towards $u$. Since $e$ and $f$ do not cross, their end vertices
occur in the cyclic order
\[
x,\ u,\ \ldots,\ u',\ \ldots,\ x'.
\]
Let $Q$ be the $u'x'$-arc of $C$ containing neither $x$ nor $u$,
and choose $e\in O$ and $f\in E$ so that
$|E(Q)|$
is minimum.

We claim that $\{x,u'\}$ is a $2$-vertex cut. Suppose otherwise. 
The two open $xu'$-arcs of $C$ are connected subgraphs whose vertex
sets partition $V(G)\setminus\{x,u'\}$. 
Since
$G-\{x,u'\}$ is connected, some edge $ab$ joins the two open $xu'$-arcs of $C$. The edge $ab$ is a chord of $C$. 
Let $a$ lie on the open $xu'$-arc containing $u$, and let $b$ lie
on the other open $xu'$-arc. The chord $ab$ cannot cross
$f=uu'$. Since $x$ and $u$ are consecutive, this forces
$a=u.$
Similarly, the chord $ub$ cannot cross $e=xx'$. Hence $b$ lies on
the $u'x'$-arc $Q$. If $b=x'$, then $ub$ is a chord whose
end vertices have opposite parity, contradicting the conclusion
preceding Claim~3. Therefore, $b$ lies in the interior of $Q$.

The chord $ub$ belongs to $E$. Moreover, the pair
$e=xx'\in O$
and 
$ub\in E$
still has distance one, since $x$ and $u$ are consecutive.
However, the arc from $b$ to $x'$ contained in $Q$ is strictly
shorter than $Q$. This contradicts the minimality of
$|E(Q)|$.

Therefore, $\{x,u'\}$ is a $2$-vertex cut. Since $x$ and $u'$ have
opposite parity, both open $xu'$-arcs of $C$ have even order. By
Claim~1, these arcs partition all vertices of
$G-\{x,u'\}$, and hence they are precisely its two components.
Theorem~\ref{lem:two-cut-reduction}(iii) now gives
$R(G)\leq\sqrt5-1,$
a final contradiction. This proves Claim~4.
\end{proof}

\medskip

By symmetry, assume that $E=\varnothing$. Let
$x_i=v_{2i-1},
y_i=v_{2i}$ 
for $i\in\mathbb Z_m$.
Every edge outside the Hamilton cycle has both end vertices among
$x_1,\ldots,x_m$. Therefore $d_G(y_i)=2$ for every
$i\in\mathbb Z_m$.

If $m=2$, then $G$ is either $C_4$ or $C_4$ with one diagonal, and
a direct calculation gives $R(G)\leq1<\sqrt5-1$. If $m\geq3$,
Lemma~\ref{lem:alternating-Hamilton-cycle} gives
$R(G)\leq\sqrt5-1$, again a contradiction. Therefore
$R(G)\leq\sqrt5-1.$
\end{proof}

\subsection{$K_{2,3}$-minor-free graphs}
\label{sec:K23}

In order to prove Theorem \ref{thm:intro-K4}, we also need to prove a lemma and a theorem as follows. 
The following equivalent block formulation follows from the
$1$-sum characterization of Yu, Shu, and Hong~\cite[Lemma~3.2]{YuShuHong2012}.

\begin{lemma}\cite[Lemma~3.2]{YuShuHong2012}
% (See~\cite[Lemma~3.2]{YuShuHong2012}.)
\label{lem:K23-block-structure}
A graph is $K_{2,3}$-minor-free if and only if each of its blocks is
either outerplanar or isomorphic to $K_4$.
\end{lemma}

\begin{theorem}
\label{thm:K23-main}
Let $G$ be a $K_{2,3}$-minor-free graph. Then
$R(G)\leq\sqrt5-1.$
\end{theorem}
\begin{proof}
Suppose, to the contrary, that the assertion is
false, and choose a counterexample $G$ of minimum order.

By minimality and Lemma~\ref{lem:general-deletion}, apply first with
$S=\varnothing$ and then with $S=\{v\}$ for a cut vertex $v$ of $G$. Then the graph $G$ is connected
and has no cut vertex. Since $R(K_1)=0$ and $R(K_2)=1<\sqrt5-1$, we have
$|V(G)|\geq3$. Hence $G$ is $2$-connected and consists of a single
block.

By Lemma~\ref{lem:K23-block-structure}, either $G$ is outerplanar or
$G\cong K_4$. In the first case, $G$ is $K_4$-minor-free, and
Theorem~\ref{thm:K4-main} gives
$R(G)\leq\sqrt5-1,$
a contradiction. In the second case,
$\Spec(K_4)=\{3,-1,-1,-1\},$
so $R(K_4)=1<\sqrt5-1$, again a contradiction. Therefore,
$R(G)\leq\sqrt5-1.$
\end{proof}

\medskip

\begin{proof}
    [Proof of Theorem \ref{thm:intro-K4}]
    It follows from Theorem \ref{thm:K4-main}, Theorem \ref{thm:K23-main}, and Lemma \ref{lem:Tm-HL-index}.
\end{proof}

\section{The proof of Theorem~\ref{thm:intro-mad-subquartic}}
\label{sec:maximum-average-degree}
In order to prove Theorem \ref{thm:intro-mad-subquartic}, we first present the following theorem and lemmas.

\begin{theorem}[Acharya, Jiang, and Zhang
{\cite[Corollary~6.3]{AcharyaJiangZhang2026}}]
\label{lem:HL-average-degree}
For every graph $G$ with average degree $\overline d(G)$, 
$R(G)\leq\sqrt{\overline d(G)}$.
\end{theorem}
\begin{lemma}
\label{lem:small-order-mad}
If $|V(G)|\leq6$, $\Delta(G)\leq4$, and
$\mad(G)\leq8/3$, then $R(G)\leq\sqrt2$.
\end{lemma}

\begin{proof}
The assertion follows from Theorem~\ref{thm:K4-main},   if $G$ is
$K_4$-minor-free. Otherwise, since $K_4$ is subcubic, $G$ contains
a subdivision $T$ of $K_4$. Writing $r=|V(T)|$, we have
$|E(T)|=r+2$, and hence
$\frac{2(r+2)}r\leq\mad(G)\leq\frac83.$ 
Thus $r\geq6$. It follows that $|V(G)|=r=6$, while
$|E(T)|=8$ and $|E(G)|\leq4|V(G)|/3=8$. Therefore $G=T$.

Hence $G$ is obtained from $K_4$ by inserting two subdivision
vertices. Up to isomorphism, there are three possibilities. Direct
calculation gives
\[
\begin{array}{c|c|c}
\text{subdivided edges}&\chi_G(x)&
(n_{\sqrt2}^+(G),n_{\sqrt2}^-(G))\\ \hline
\text{the same edge twice}
&(x-1)(x+1)(x^2-x-5)(x^2+x-1)&(1,2)\\
\text{two adjacent edges}
&(x^2+x-1)(x^4-x^3-6x^2+3x+1)&(1,2)\\
\text{two independent edges}
&x^2(x^2-2x-2)(x^2+2x-2)&(1,1)
\end{array}
\]
For the factor  $p(x)=x^4-x^3-6x^2+3x+1$ in the second row,
$p(-3),p(-2),\ldots,p(3)=46,-5,-6,1,-2,-9,10;$
hence its four roots lie in $(-3,-2)$, $(-1,0)$, $(0,1)$, and
$(2,3)$. Thus in every case
\[
n_{\sqrt2}^\pm(G)\leq2
=\left\lfloor\frac{|G|-1}{2}\right\rfloor,
\]
and the result follows from
Lemma~\ref{lem:count-characterization}.
\end{proof}

\begin{lemma}
\label{lem:basic-critical-reductions}
Let $G$ be a  graph of the minimum order that satisfies 
$\Delta(G)\leq4$, $\mad(G)\leq8/3$, and $R(G)>\sqrt2$.
If $|V(G)|\geq7$, then $G$ is $2$-connected, no two
$2$-vertices are adjacent, and every $3$-vertex has at most one
$2$-neighbor.
\end{lemma}

\begin{proof}
Let $n=|V(G)|$ and $a=\sqrt2$. Note that for every subgraph $H$ of $G$, 
$\Delta(H)\leq4$ and $\mad(H)\leq8/3$. Thus by the minimality of $|V(G)|$, we have $R(H)\le a$ for every proper subgraph $H$ of $G$. 
If $G$ were
disconnected, or if it had a cut vertex, minimality together with
Lemma~\ref{lem:general-deletion} would give $R(G)\leq a$, a contradiction. Thus $G$
is connected and has no cut vertex. Since $n\geq7$, the graph $G$
is $2$-connected and the minimum degree of $G$ is at least $2$.

Suppose that $u$ and $v$ are adjacent $2$-vertices. Let $S$ be the
set of their other neighbors and let $s=|S|\leq2$. Then
$G-S=P_2\cup H$, where $|H|=n-s-2$ and $H$ is nonempty. By
minimality, $R(H)\leq a$. Since
$\Spec(P_2)=\{1,-1\}\subseteq[-a,a]$, the graph $P_2$ contributes
no eigenvalue to $n_a^\pm(G-S)$. Lemmas
\ref{lem:count-characterization} and
\ref{lem:count-interlacing} therefore give
\[
n_a^\pm(G)
\leq s+\left\lfloor\frac{n-s-3}{2}\right\rfloor
\leq\left\lfloor\frac{n-1}{2}\right\rfloor,
\]
contrary to $R(G)>a$. Hence no two $2$-vertices are adjacent.

Now suppose that a $3$-vertex $v$ has two $2$-neighbors
$u_1,u_2$. These vertices are non-adjacent by the preceding
paragraph. Let $S$ consist of the other neighbors of $u_1,u_2$ and
the third neighbor of $v$, and let $s=|S|\leq3$. Then
$G-S=P_3\cup H$, where $|H|=n-s-3\geq1$. Again $R(H)\leq a$ by
minimality. Since
$\Spec(P_3)=\{\sqrt2,0,-\sqrt2\},$ 
$P_3$ contributes no eigenvalue to $n_a^\pm(G-S)$. Therefore,
\[
n_a^\pm(G)
\leq s+\left\lfloor\frac{n-s-4}{2}\right\rfloor
\leq\left\lfloor\frac{n-1}{2}\right\rfloor,
\]
again contradicting Lemma~\ref{lem:count-characterization}.
\end{proof}

\begin{proof}[proof of Theorem \ref{thm:intro-mad-subquartic}.]
Set $a=\sqrt2$ and suppose that $G$ is a counterexample of minimum
order. Let $n=|V(G)|$. By Lemma~\ref{lem:small-order-mad},
$n\geq7$, so we apply Lemma~\ref{lem:basic-critical-reductions}; hence we have that $G$ is $2$-connected, no two
$2$-vertices are adjacent, and every $3$-vertex has at most one
$2$-neighbor.

For $i\in\{2,3,4\}$, let $V_i$ be the set of $i$-vertices
(i.e., vertices of degree $i$) and let
$n_i=|V_i|$. Since the minimum degree of $G$ is at least $2$, we have
$n=n_2+n_3+n_4$. Moreover,
$
2n_2=e(V_2,V_3)+e(V_2,V_4)\leq n_3+4n_4.
$
Consequently,
\[
\begin{aligned}
3\sum_{v\in V(G)}d_G(v)-8n
&=3(2n_2+3n_3+4n_4)-8(n_2+n_3+n_4)\\
&=-2n_2+n_3+4n_4\\
&\geq0.
\end{aligned}
\]
Thus
$\overline d(G)
=\frac{1}{n}\sum_{v\in V(G)}d_G(v)
\geq\frac83.
$

Since
$\overline d(G)\leq\mad(G)\leq\frac83,$
we have equality throughout. In particular,
$2n_2=n_3+4n_4.$
Furthermore,
$2n_2=e(V_2,V_3)+e(V_2,V_4),$
where
$
e(V_2,V_3)\leq n_3,
e(V_2,V_4)\leq4n_4.
$
Since equality holds in their sum, equality holds in both
inequalities. Hence every $3$-vertex has exactly one $2$-neighbor,
and every $4$-vertex has exactly four $2$-neighbors.
Therefore no $4$-vertex is adjacent to a vertex of
$V_3\cup V_4$, while every $3$-vertex has exactly two neighbors in
$V_3$. It follows that $G[V_3]$ is a disjoint union of cycles.
Moreover,
$n_3=2n_2-4n_4$
is even, so we may write $n_3=2r$. 

Let $q$ be the number of odd cycles in
$G[V_3]$. Since the total number $n_3$ of vertices on these cycles
is even, $q$ is even. Choose a minimum vertex cover on each cycle,
let $X$ be their union. Then
$|X|=\frac{n_3+q}{2}=r+\frac q2.$

The set $V_3\setminus X$ is independent. Every vertex in
$V_3\setminus X$ has exactly one neighbor in $V_2$, every vertex of
$V_2$ has at most two neighbors in $V_3\setminus X$, and $V_2$ is
independent. Let $S=X\cup V_4$. Hence every component of $G-S$ is isomorphic to
$K_1$, $P_2$, or $P_3$. Since
$\Spec(P_2)=\{1,-1\},
\Spec(P_3)=\{\sqrt2,0,-\sqrt2\},$ every eigenvalue of $G-S$ belongs to $[-a,a]$. Thus
$n_a^+(G-S)=n_a^-(G-S)=0.$

From $2n_2=n_3+4n_4$, we have $n_2=r+2n_4$. Setting
$t=r+n_4$ gives
$n=3t$ and $ |S|=t+\frac q2.$ 
Every odd cycle has length at least three, so
$3q\leq n_3=2r\leq2t$. Hence $q<t$. Since $q$ is even, we have
$q\leq t-1$ when $t$ is odd and $q\leq t-2$ when $t$ is even.
Therefore,
$|S|\leq\left\lfloor\frac{3t-1}{2}\right\rfloor
=\left\lfloor\frac{n-1}{2}\right\rfloor.$
Lemma~\ref{lem:count-interlacing} now gives
\[n_a^\pm(G)
\leq n_a^\pm(G-S)+|S|
\leq\left\lfloor\frac{n-1}{2}\right\rfloor,
\]
contradicting Lemma~\ref{lem:count-characterization}.
\end{proof}

\section{The proof of Theorem~\ref{thm:intro-average-degree}}
\label{sec:average-degree}

Following the approach used in the proof of \cite[Theorem~1.1]{AcharyaJiangZhang2026}, we construct a suitable  polynomial and employ the first four spectral moments to prove Theorem \ref{thm:intro-average-degree}.
\begin{proof}[proof of Theorem \ref{thm:intro-average-degree}.]
Let $n=|V(G)|$, $m=|E(G)|$, and $a=\sqrt{d-2}$. Consider the polynomial $f$ defined by
\[
f(x)=(r-x)(x+a)^2(x+d),
\text{ where }
r=d+\frac{2a^2}{d+2a}.
\]
Since $r-d=2a^2/(d+2a)$ and
$r-a=d(d+a)/(d+2a)$, direct substitution gives
$f(a)=f(d)=M$, where
\[M=\frac{4da^2(d+a)^2}{d+2a}.\]
Moreover, direct factorization yields
\[
f(x)-M=(d-x)(x-a)Q(x), 
\text{ where }
Q(x)=x^2+\frac{d^2+5ad+4a^2}{d+2a}x
+\frac{a(3d^2+6ad+2a^2)}{d+2a}.
\]
All coefficients of $Q$ are positive, and hence $Q(x)>0$ for
$x\geq a$. Since $r>d$, the three factors $r-x$, $(x+a)^2$, and
$x+d$ are nonnegative on $[-d,d]$. Consequently, $f(x)\geq0$ on
$[-d,d]$, while $f(x)\geq M$ for $a<x\leq d$, with strict
inequality for $a<x<d$.

\medskip

We divide the rest of the proof into the following two parts.

(i) We first assume that $n=2h$. 
Let $A$ be the adjacency matrix of $G$, with eigenvalues
$\lambda_1\geq\cdots\geq\lambda_n$. Since $\Delta(G)\leq d$, all
the eigenvalues lie in $[-d,d]$. Expanding $f$ and using
$a^2=d-2$, we obtain
\[
f(x)=-x^4-\frac{2a(a+d)}{d+2a}x^3
 +(d-1)(d+2)x^2
 +\frac{2a(a+d)(a^2+ad+d^2)}{d+2a}x+da^2r.
\]
Let $t(G)$ denote the number of triangles in $G$. 
Since $\operatorname{tr}(A)=0$,
$\operatorname{tr}(A^2)=2m$, and
$\operatorname{tr}(A^3)=6t(G)\geq0$, it follows that
\[
\sum_{i=1}^n f(\lambda_i)
\leq-\operatorname{tr}(A^4)
+2(d-1)(d+2)m+da^2rn.
\]

Let $c_4(G)$ denote the number of $4$-cycles in $G$. By the classical fourth spectral moment formula (see \cite[pp.~86--87]{GutmanDas2004}), we have
\[\operatorname{tr}(A^4)
=2\sum_{v\in V(G)}d_G(v)^2-2m+8c_4(G).
\]
For every integer $0\leq k\leq d$, the inequality
$(k-d+2)(k-d+1)\geq0$ gives
$k^2\geq(2d-3)k-(d-2)(d-1)$. Applying this inequality to all
vertex degrees and discarding the nonnegative term $8c_4(G)$,
we obtain
\[
\operatorname{tr}(A^4)
\geq(8d-14)m-2(d-2)(d-1)n.
\]
Consequently,
\[
\sum_{i=1}^n f(\lambda_i)
\leq2(d^2-3d+5)m
+a^2\bigl(dr+2(d-1)\bigr)n.
\]

Recall that for every integer $d\geq3$, $\tau_d$ is defined by 
\[
\tau_d:=\frac{(d-2)(d^2-2d+2)}{d^2-3d+5}.
\]
Since $2m/n=\overline d(G)\leq\tau_d$ and $a^2=d-2$, we have
$2(d^2-3d+5)m\leq a^2(d^2-2d+2)n$. Moreover,
$d+r=2(d+a)^2/(d+2a)$. Therefore,
\[
\sum_{i=1}^n f(\lambda_i)
\leq a^2d(d+r)n
=\frac{2da^2(d+a)^2}{d+2a}n
=\frac{nM}{2}=hM.
\]

Suppose that $\lambda_h(G)>a$. Then at least $h$ eigenvalues lie in $(a,d]$, and each contributes at least $M$ to the preceding sum $\sum_{i=1}^n f(\lambda_i)$. 
Moreover, the multiplicity of $d$ is less than $h$. 
Indeed, every connected component having $d$ as an eigenvalue is $d$-regular, has at least $d+1$ vertices. and contributes $d$ with multiplicity one, by the Perron-Frobenius theorem. 
Thus the total multiplicity of $d$ is at most $n/(d+1)<n/2=h$. Hence at least one of the first $h$
eigenvalues lies in $(a,d)$ and contributes strictly more than
$M$. Since $f$ is nonnegative on $[-d,d]$, this contradicts
$\sum_i f(\lambda_i)\leq hM$. 
Therefore, $\lambda_h(G)\leq a$.

Now suppose that $n=2h-1$ and set $G'=G\cup K_1$. Then
$|V(G')|=2h$, $\Delta(G')\leq d$, and
$\overline d(G')\leq\overline d(G)\leq\tau_d$. 
Thus by the preceding argument, $\lambda_h(G')\leq a$.  
Since $G$ is obtained from $G'$ by
deleting the isolated vertex, Theorem \ref{thm:interlacing} (the eigenvalue interlacing theorem) gives
$\lambda_h(G)\leq\lambda_h(G')\leq a$.

\medskip

(ii) Suppose that $G$ is additionally triangle-free. Then
$\operatorname{tr}(A^3)=0$. First assume that $n$ is even. Applying the same calculation as in (i) to
$-\lambda_1,\ldots,-\lambda_n$ gives
$\sum_i f(-\lambda_i)\leq hM$. The multiplicity of $-d$ is also
less than $h$, since every connected component having $-d$ as an
eigenvalue is $d$-regular and bipartite, has at least $d+1$
vertices, and contributes $-d$ with multiplicity one, by the Perron-Frobenius theorem. Thus
$\lambda_\ell(G)<-a$ would imply that at least $h$ of the numbers
$-\lambda_i$ lie in $(a,d]$, with at least one lying in $(a,d)$,
again contradicting $\sum_i f(-\lambda_i)\leq hM$. 
Hence, $\lambda_\ell(G)\geq-a$, and therefore $R(G)\leq a$ when $n$ is even.

Finally, suppose that $n=2h-1=2\ell-1$ and also set $G'=G\cup K_1$. Then
$|V(G')|=2h$, $\Delta(G')\leq d$, and 
$\overline d(G')\leq\overline d(G)\leq\tau_d$.
Moreover, $G'$ is triangle-free, since $G$ is triangle-free. 
Thus by the preceding argument, $R(G')\leq a$. 
Again Theorem \ref{thm:interlacing} (the eigenvalue interlacing theorem) gives
\[a\geq
\lambda_h(G')\geq\lambda_h(G)=\lambda_\ell(G)
\geq\lambda_{h+1}(G')\geq -a.
\]
Thus, $R(G)\leq a$, which completes the proof.
\end{proof}

We derive the following corollary by applying Theorem~\ref{thm:intro-average-degree} with $d=4$, noting that $\tau_4=20/9$.
\begin{corollary}
\label{cor:subquartic-small-average-degree}
Let $G$ be a simple graph satisfying $\Delta(G)\leq4$ and
$\overline d(G)\leq20/9$. Then we have
$\lambda_h(G)\leq\sqrt2$. Moreover, if $G$ is triangle-free, then
$R(G)\leq\sqrt2$.
\end{corollary}

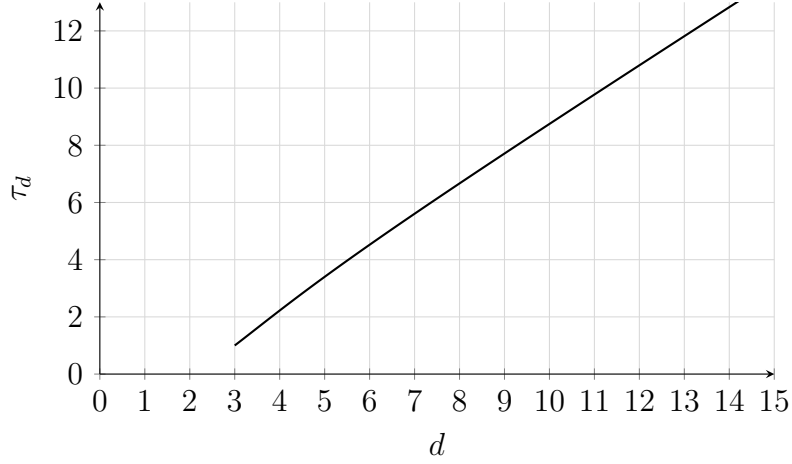
\begin{figure}[htbp]
\centering
\begin{tikzpicture}
\begin{axis}[
    width=10.5cm,
    height=6.5cm,
    xlabel={$d$},
    ylabel={$\tau_d$},
    xmin=0,
    xmax=15,
    ymin=0,
    ymax=13,
    domain=3:15,
    samples=300,
    axis lines=left,
    grid=both,
    major grid style={gray!30},
    minor grid style={gray!15},
    xtick={0,1,2,3,4,5,6,7,8,9,10,11,12,13,14,15},
    ytick={0,2,4,6,8,10,12}
]
\addplot[thick]
{((x-2)*(x^2-2*x+2))/(x^2-3*x+5)};
\end{axis}
\end{tikzpicture}
\caption{The dependence of
$\tau_d=\dfrac{(d-2)(d^2-2d+2)}{d^2-3d+5}$ on $d$ for integers
$d\geq3$.}
\label{fig:tau-d}
\end{figure}

\begin{remark}
The dependence of the threshold $\tau_d$ on $d$ is illustrated in
Figure.~\ref{fig:tau-d}. The figure shows that the threshold is close to
$d-1$ for large $d$. In fact, it admits the following asymptotic
expansion:
\[
\tau_d=d-1-\frac{2d-1}{d^2-3d+5}
=d-1-\frac2d+O\left(\frac1{d^2}\right)
\]
as $d\to\infty$.
\end{remark}

\bibliographystyle{plainnat}
\bibliography{references_alphabetical}
\end{document}